\documentclass[a4paper,12pt]{article}     
\usepackage{amsmath,amscd,amssymb}        
\usepackage{latexsym}                     
\usepackage[english, german]{babel}                

\usepackage{theorem}                 

\newtheorem{theorem}{Theorem}

\newtheorem{lemma}{Lemma}

\newenvironment{definition}
{\smallskip\noindent{\bf Definition\/}:}{\smallskip\par}

\newenvironment{remark}
{\smallskip\noindent{\bf Remark\/}.}{\smallskip\par}

\newenvironment{proof}{\begin{ProofwCaption}{Proof}}{\end{ProofwCaption}}
\newenvironment{proof*}[1]{\begin{ProofwCaption}{{#1}}}{\end{ProofwCaption}}
\newenvironment{ProofwCaption}[1]%
  {\addvspace\theorempreskipamount \noindent{\it #1.}\rm}%
  {\qed \par \addvspace\theorempostskipamount}
\newcommand{\qedsymbol}{{\rm $\Box$}}
\newcommand{\qed}{\hfill\qedsymbol}

\newcommand{\CC}{{\mathbb C}}

\newcommand{\RR}{{\mathbb R}}
\newcommand{\ZZ}{{\mathbb Z}}

\newcommand{\eps}{\varepsilon}

\title{Orbifold zeta functions for dual invertible polynomials}
\author{Wolfgang Ebeling and Sabir M.~Gusein-Zade
\thanks{Partially supported by DFG (Mercator fellowship, Eb 102/8-1), RFBR-13-01-00755
and NSh-5138.2014.1.
AMS 2010 Math. Subject Classification: 14J33, 14R20, 57R18, 58K10.
}
}
\date{}

\begin{document}
\selectlanguage{english}

\maketitle

\begin{abstract}
An invertible polynomial in $n$ variables is a quasihomogeneous polynomial consisting of $n$ monomials
so that the weights of the variables and the quasi-degree are well defined. 
In the framework of the construction of mirror symmetric orbifold Landau--Ginzburg models,
P.~Berg\-lund, T.~H\"ubsch and M.~Henningson considered a pair $(f,G)$ consisting of
an invertible polynomial $f$ and an abelian group $G$ of its symmetries together with
a dual pair $(\widetilde{f}, \widetilde{G})$. Here we study the reduced orbifold
zeta functions of dual pairs $(f,G)$ and $(\widetilde{f}, \widetilde{G})$ and show
that they either coincide or are inverse to each other depending on the number $n$ of variables.
\end{abstract}

\section*{Introduction} 
P.~Berglund and T.~H\"ubsch \cite{BH1} proposed a method to construct some 
mirror symmetric pairs of manifolds. Their construction involves a polynomial $f$
of a special form, a so called  {\em invertible} one, and its
{\em Berglund--H\"ubsch transpose} $\widetilde f$. In \cite{BH1} these polynomials
appeared as potentials of Landau--Ginzburg models. 
This construction was
generalized in \cite{BH2} to orbifold Landau--Ginzburg models described by pairs
$(f, G)$, where $f$ is an  invertible polynomial and $G$ is a (finite)
abelian group of symmetries of $f$. For a pair $(f, G)$ one defines the dual pair
$(\widetilde{f}, \widetilde{G})$. In \cite{BH2, KY}, there were described some
symmetries between invariants of the pairs $(f, G)$ and $(\widetilde{f}, \widetilde{G})$
corresponding to the orbifolds defined by the equations $f=0$ and $\widetilde{f}=0$ in
weighted projective spaces.  Some duality (symmetry) properties of 
the singularities defined by $f$ and $\widetilde f$ were observed in \cite{MMJ, BLMS, ET2, Taka}.
In particular, in \cite{MMJ} it was shown that the reduced orbifold Euler characteristics
of the Milnor fibres of $f$ and $\widetilde f$ with
the actions of the groups $G$ and $\widetilde G$ respectively coincide
up to sign.

Here we consider the (reduced) orbifold zeta function defined in \cite{ET2}. One can say that it collects
an information about the eigenvalues of monodromy operators modified by so called age (or fermion) shifts.
We show that the (reduced) orbifold zeta functions of Berglund-H\"ubsch-Henningson dual pairs
$(f, G)$ and $(\widetilde{f}, \widetilde{G})$ either coincide or are inverse to each other depending on
the number $n$ of variables. This is a refinement of the above mentioned result of \cite{MMJ}
which means that the degrees of these zeta functions coincide up to sign.

\section{Invertible polynomials}\label{Invertible}
A quasihomogeneous polynomial $f$ in $n$ variables is called 
{\em invertible} (see \cite{Kreuzer}) if it contains $n$ monomials, i.e.\ it is of the form 
\begin{equation}\label{inv} 
f(x_1, \ldots, x_n)=\sum\limits_{i=1}^n a_i \prod\limits_{j=1}^n x_j^{E_{ij}} 
\end{equation} 
for $a_i\in\CC^\ast=\CC\setminus\{0\}$, and the matrix 
$E=(E_{ij})$ (with non-negative integer entries) is non-degenerate: $\det E\ne 0$. 
Without loss of generality we may assume that $a_i=1$ 
for $i=1, \ldots, n$ and that $\det E>0$. 

The {\em Berglund-H\"ubsch transpose} $\widetilde{f}$ of the invertible polynomial (\ref{inv})
is
$$ 
\widetilde{f}(x_1, \ldots, x_n)=\sum\limits_{i=1}^n a_i \prod\limits_{j=1}^n x_j^{E_{ji}}\,, 
$$ 
i.e.\ it is defined by the transpose $E^T$ of the matrix $E$.

The (diagonal) {\em symmetry group} of the invertible polynomial $f$ is the group $G_f$
of diagonal linear transformations of $\CC^n$ preserving $f$:
$$ 
G_f=\{(\lambda_1, \ldots, \lambda_n)\in (\CC^*)^n: 
f(\lambda_1 x_1, \ldots, \lambda_n x_n)= f(x_1, \ldots, x_n)\}\,.
$$  
This group is finite and its order $\vert G_f\vert$ is equal to $d=\det E$ \cite{Kreuzer,ET2}.
The polynomial $f$ is quasihomogeneous with respect to the rational weights $q_1$, \dots, $q_n$
defined by the equation
$$
E(q_1, \ldots, q_n)^T=(1, \ldots, 1)^T\,,
$$
i.e.\
$$
f(\exp(2\pi i q_1\tau) x_1, \ldots, \exp(2\pi i q_n\tau) x_n)= \exp(2\pi i\tau) f(x_1, \ldots, x_n)\,.
$$

The {\em Milnor fibre} of the polynomial $f$ is the manifold
$$
V_f=\{(x_1, \ldots, x_n)\in\CC^n: f(x_1, \ldots, x_n)=1\}\,.
$$
The {\em monodromy transformation} (see below) is induced by the element
$$
g_0=(\exp(2\pi i q_1), \ldots, \exp(2\pi i q_n))\in G_f\,.
$$
(In \cite{Krawitz} the element $g_0$ is called the ``exponential grading operator''.)

For a finite abelian group $G$, let $G^*={\rm Hom\,}(G,\CC^*)$ be its group of characters. 
(The groups $G$ and $G^*$ are isomorphic, but not in a canonical way.) One can show
that the symmetry group $G_{\widetilde{f}}$ of the
Berglund-H\"ubsch transpose $\widetilde{f}$ of an invertible polynomial $f$ is
canonically isomorphic to $G_f^*$ (see, e.g., \cite{BLMS}). The duality between $G_f$ and $G_{\widetilde{f}}$
is defined by the pairing
$$
\langle\underline{\lambda}, \underline{\mu}\rangle_E=\exp(2\pi i(\underline{\alpha}, \underline{\beta})_E)\,,
$$
where $\underline{\lambda} = (\exp(2\pi i\,\alpha_1), \ldots, \exp(2\pi i\,\alpha_n))\in G_{\widetilde{f}}$, 
$\underline{\mu} = (\exp(2\pi i\,\beta_1), \ldots, \exp(2\pi i\,\beta_n))\in G_f$,
$\underline{\alpha}=(\alpha_1, \ldots, \alpha_n)$,
$\underline{\beta}=(\beta_1, \ldots, \beta_n)$,
$$
(\underline{\alpha}, \underline{\beta})_E:=(\alpha_1, \ldots, \alpha_n)E(\beta_1, \ldots, \beta_n)^T
$$
(see \cite{BLMS}).

\begin{definition}
 (\cite{BH2}) For a subgroup $H\subset G_f$ its {\em dual} $\widetilde H\subset G_{\widetilde{f}}=G_f^*$
 is the kernel of the natural map $i^*:G_f^*\to H^*$ induced by the inclusion $i:H\hookrightarrow G_f$.
\end{definition}

One can see that $\vert H\vert\cdot\vert\widetilde H\vert=\vert G_f\vert=\vert G_{\widetilde{f}}\vert$.

\begin{lemma}\label{SL}
Let $G_{f,0}=\langle g_0\rangle$ be the subgroup of $G_f$ generated by the monodromy transformation. One has
$\widetilde{G_{f,0}}= G_{\widetilde{f}} \cap {\rm SL\,}(n,\CC)$.
\end{lemma}

\begin{proof}
 For $\underline{\lambda} = (\exp(2\pi i\,\alpha_1), \ldots, \exp(2\pi i\,\alpha_n))\in G_{\widetilde{f}}$,
 $\langle\underline{\lambda}, g_0\rangle_E=1$ if and only if
 $$
 (\alpha_1, \ldots, \alpha_n)E(q_1, \ldots, q_n)^T\in \ZZ\,.
 $$
One has $E(q_1, \ldots, q_n)^T=(1, \ldots, 1)^T$.
Therefore $(\alpha_1, \ldots, \alpha_n)(1, \ldots, 1)^T\in \ZZ$, i.e.\
$\sum\limits_i \alpha_i\in\ZZ$, $\prod\limits_i \lambda_i=1$. This means that
$\underline{\lambda}\in {\rm SL\,}(n,\CC)$.
\end{proof}

\section{Orbifold zeta function}\label{sect-orbifold-zeta}
The {\em zeta function} $\zeta_h(t)$ of a (proper, continuous) transformation
$h:X\to X$ of a topological space $X$ is the rational function defined by
\begin{equation} \label{Defzeta}
\zeta_h(t)=\prod\limits_{q\ge 0} 
\left(\det({\rm id}-t\cdot h^{*}{\rm \raisebox{-0.5ex}{$\vert$}}{}_{H^q_c(X;\RR)})\right)^{(-1)^q}\,,
\end{equation}
where $H^q_c(X;\RR)$ denotes the cohomology with compact support.
The degree of the zeta function $\zeta_h(t)$, i.e.\ the degree of the numerator minus the degree
of the denominator, is equal to the Euler characteristic $\chi(X)$ of the space $X$
(defined via cohomology with compact support).

\begin{remark}\label{free}
 If a transformation $h:X\to X$ defines on $X$ a free action of the cyclic group of order $m$,
 (i.e.\ if $h^m(x)=x$ for $x\in X$, $h^k(x)\ne x$ for $0<k<m$, $x\in X$), then 
 $\zeta_h(t)=(1-t^m)^{\chi(X)/m}$.
\end{remark}

The {\em monodromy zeta function},
i.e.\ the zeta function of a monodromy transformation is of the form
$\prod\limits_{m\ge1}(1-t^m)^{s_m}$, where $s_m$ are integers such that only finitely many of them
are different from zero (see, e.g., \cite{AGV}). In particular, all roots and/or poles of the monodromy
zeta function are roots of unity.

The {\em orbifold} (monodromy) {\em zeta function} was essentially defined in \cite{ET2}.

Let $G$ be a finite group acting on the space $\CC^n$ by a representation and let
$f:(\CC^n,0)\to(\CC,0)$ be a $G$-invariant germ of a holomorphic function. The Milnor fibre $V_f$
of the germ $f$ is the manifold $\{f=\eps\}\cap B_\delta^{2n}$, where $B_\delta^{2n}$ is the ball
of radius $\delta$ centred at the origin in $\CC^n$, $0<\vert\eps\vert\ll\delta$, $\delta$ is
small enough. One may assume the monodromy transformation $h_f$ of the germ $f$ \cite{AGV}
to be $G$-invariant. For an element $g\in G$, its {\em age} \cite{Ito-Reid}
(or fermion shift number: \cite{Zaslow}) is defined by ${\rm age\,}(g):=\sum\limits_{i=1}^n\alpha_i$,
where in a certain basis in $\CC^n$ one has
$g={\rm diag\,}(\exp(2\pi i\,\alpha_1), \ldots, \exp(2\pi i\,\alpha_n))$
with $0\le\alpha_i < 1$.

\begin{remark} The map $\exp(2\pi i\,{\rm age\,}(\cdot)) : G \to \CC^\ast$ is a group homomorphism.
If $f$ is an invertible polynomial and $G$ is the group $G_f$ of its symmetries, then it is
an element of $G_f^\ast=G_{\widetilde{f}}$. 
\end{remark}

For a rational function $\varphi(t)$ of the form $\prod\limits_i(1-\alpha_i t)^{r_i}$
with only finitely many of the exponents $r_i\in\ZZ$ different from zero, its {\em  $g$-age shift}
is defined by
$$
\left(\varphi(t)\right)_g=\prod\limits_i(1-\alpha_i \exp(-2\pi i\,{\rm age\,}(g)) t)^{r_i}\,,
$$
i.e.\ all its roots and/or poles are multiplied by $\exp(2\pi i\,{\rm age\,}(g))\in\CC^*$.

Let ${\rm Conj\,}G$ be the set of conjugacy classes of elements of $G$. For a class $[g]\in{\rm Conj\,}G$,
let $g\in G$ be a representative of it. Let $C_G(g)$ be the centralizer of the element $g$ in $G$.
Let $(\CC^n)^g$ be the fixed point set of the element $g$,
let $V_f^g=V_f\cap (\CC^n)^g$ be the corresponding part of the Milnor fibre, and let
$\widehat{V}_f^g= V_f^g/C_G(g)$ be the corresponding quotient space (the ``twisted sector'' in terms of
\cite{Chen_Ruan}). One may assume that the monodromy transformation preserves $V_f^g$ for each $g$.
Let $\widehat{h}_f^g:\widehat{V}_f^g\to \widehat{V}_f^g$ be the corresponding map (monodromy)
on the quotient space.
Its zeta function $\zeta_{\widehat{h}_f^g}(t)$ depends only on the conjugacy class of $g$.

\begin{definition}
 The {\em orbifold zeta function} of the pair $(f,G)$ is defined by
 \begin{equation}\label{orbifold-zeta}
  \zeta^{{\rm orb}}_{f,G}(t)= \prod\limits_{[g]\in {\rm Conj\,}G}\left(\zeta_{\widehat{h}_f^g}(t)\right)_g\,.
 \end{equation}
\end{definition}

One can see that the degree of $\zeta^{{\rm orb}}_{f,G}(t)$ is equal to the orbifold Euler
characteristic of $(V_f, G)$ (see, e.g., \cite{HH}, \cite{MMJ}).

For an abelian G, $\widehat{V}_f^g=V_f^g/G$ and the product in (\ref{orbifold-zeta}) runs
over all elements $g\in G$.

\begin{definition}
  The {\em reduced orbifold zeta function} $\overline{\zeta}^{{\rm orb}}_{f,G}(t)$ is defined by
  $$
  \overline{\zeta}^{{\rm orb}}_{f,G}(t)=
  \zeta^{{\rm orb}}_{f,G}(t)\left/\prod\limits_{[g]\in {\rm Conj\,}G}(1-t)_g \right.
  $$
  (cf. (\ref{orbifold-zeta})).
\end{definition}

Now let $G$ be abelian. One can assume that the action of $G$ on $\CC^n$ is diagonal and therefore
it respects the decomposition of $\CC^n$ into the coordinate tori.
For a subset $I\subset I_0=\{1, 2, \ldots, n\}$, let 
$$
(\CC^*)^I:= \{(x_1, \ldots, x_n)\in \CC^n: x_i\ne 0 {\rm \ for\ }i\in I, x_i=0 {\rm \ for\ }i\notin I\}
$$
be the corresponding coordinate torus. Let $V_f^I=V_f\cap (\CC^*)^I$.
One has $V_f=\coprod\limits_{I\subset I_0}V_f^I$.
Let $G^I\subset G$ be the isotropy subgroup of the action of $G$ on the torus $(\CC^*)^I$.
(All points of the torus $(\CC^*)^I$ have one and the same isotropy subgroup.)
The monodromy transformation $h_f$ is assumed to respect the decomposition of the Milnor fibre
$V_f$ into the parts $V_f^I$. Let $h_f^I$ and $\widehat{h}_f^I$ be the corresponding (monodromy)
transformations of $V_f^I$ and $V_f^I/G$ respectively. One can define in the same way as above
the orbifold zeta function corresponding to the part $V_f^I$ of the Milnor fibre:
\begin{equation}\label{zeta-part1}
 \zeta^{{\rm orb},I}_{f,G}(t)= \prod\limits_{g\in G}\left(\zeta_{\widehat{h}_f^{I,g}}(t)\right)_g\,.
 \end{equation}
One has
$$
 \zeta^{{\rm orb}}_{f,G}(t)= \prod\limits_{I\subset I_0}\zeta^{{\rm orb},I}_{f,G}(t)\,.
$$
Since the isotropy subgroups of all points of $(\CC^*)^I$ are the same (equal to $G^I$),
the equation (\ref{zeta-part1}) reduces to
\begin{equation}\label{zeta-part2}
 \zeta^{{\rm orb},I}_{f,G}(t)= \prod\limits_{g\in G^I}\left(\zeta_{\widehat{h}_f^I}(t)\right)_g\,.
 \end{equation}
The (monodromy) zeta function $\zeta_{\widehat{h}_f^I}(t)$ has the form
$\prod\limits_{m\ge 0}(1-t^m)^{s_m}$ with only a finite number of the exponents $s_m$ different from zero.
Let us compute $\prod\limits_{g\in G^I}\left(1-t^m\right)_g$.
\begin{lemma}\label{1-tm}
 One has
$$
\prod\limits_{g\in G^I}\left(1-t^m\right)_g=
\left(1-t^{{\rm lcm\,}(m,k)}\right)^{\frac{m\vert G^I\vert}{{\rm lcm\,}(m,k)}}\,,
$$
where $k=\vert G^I/ G^I \cap {\rm SL}(n,\CC)\vert$, ${\rm lcm\,}(\cdot,\cdot)$ denotes the least common multiple.
\end{lemma}

\begin{proof}
 The roots of the binomial $(1-t^m)$ are all the $m$th roots of unity. The map
 $\exp(2\pi i\,{\rm age}(\cdot)): G^I\to \CC^*$ is a group homomorphism. Its kernel coincides
 with $G^I\cap {\rm SL}(n,\CC)$. Therefore its image consists of all the $k$th roots of unity
 (each one corresponds to $\vert G^I\cap {\rm SL}(n,\CC)\vert$ elements of $G^I$). Thus the roots
 of $\prod\limits_{g\in G^I}\left(1-t^m\right)_g$ are all the roots of unity of degree
 ${\rm lcm\,}(m,k)$ with equal multiplicities. This means that
 $\prod\limits_{g\in G^I}\left(1-t^m\right)_g=(1-t^{{\rm lcm\,}(m,k)})^s$. The exponent $s$
 is determined by the number of roots.
\end{proof}

\section{Orbifold zeta functions for invertible polynomials}\label{main}
Let $(f,G)$ be a pair consisting of an invertible polynomial $f$ in $n$ variables
and a group $G\subset G_f$ of its (diagonal) symmetries and let $(\widetilde{f},\widetilde{G})$
be the Berglund--H\"ubsch--Henningson dual pair ($\widetilde{G}\subset G_{\widetilde{f}}$).
(We do not assume that the invertible polynomials are non-degenerate, i.e.\ that they have
isolated critical  points at the origin.)

\begin{theorem}
 One has
 \begin{equation}\label{main-eq}
  \overline{\zeta}^{{\rm orb}}_{\widetilde{f},\widetilde{G}}(t)=\left(
  \overline{\zeta}^{{\rm orb}}_{f,G}(t)  
  \right)^{(-1)^n}\,.
 \end{equation}
 \end{theorem}

\begin{proof} We use the notations from Section~\ref{sect-orbifold-zeta}. One has
\begin{equation}\label{product}
 \overline{\zeta}^{{\rm orb}}_{f,G}(t)=\prod\limits_{I\subset I_0} 
  \zeta^{{\rm orb}, I}_{f,G}(t) \left/\prod\limits_{[g]\in {\rm Conj\,}G}(1-t)_g\,. \right.
\end{equation}

 Let $\ZZ^n$ be the lattice of monomials in the variables $x_1$, \dots, $x_n$
($(k_1, \ldots, k_n)\in \ZZ^n$ corresponds to the monomial $x_1^{k_1}\cdots x_n^{k_n}$) 
and let $\ZZ^I:=\{(k_1, \ldots, k_n)\in \ZZ^n: k_i=0 \mbox{ for }i\notin I\}$. 
For a polynomial $F$ in the variables $x_1$, \dots, $x_n$, let $\mbox{supp\,} F\subset \ZZ^n$
be the set of monomials (with non-zero coefficients) in $F$.

The elements of the subgroup $G_{f,0}\cap G_f^I$ act on $V_f^I$ trivially. The monodromy
transformation defines a free action of the cyclic group $G_{f,0}/(G_{f,0}\cap G_f^I)$ on $V_f^I$.
Therefore the monodromy transformation on $V_f^I/G$ defines an action of the cyclic group
$G_{f,0}/\left(G_{f,0}\cap (G+G_f^I)\right)$ which is also free. According to the remark
at the beginning of Section~\ref{sect-orbifold-zeta}, the zeta function is given by
\begin{equation}\label{zeta-quotient}
 \zeta_{\widehat{h}_f^I}(t)=(1-t^{m_I})^{s_I}\,,
\end{equation}
where
$$
m_I=\vert G_{f,0}/\left(G_{f,0}\cap (G+G_f^I)\right)\vert=
\frac{\vert G+G_f^I+G_{f,0}\vert}{\vert G+G_f^I\vert}\,,
$$
$s_I=\chi(V_f^I/G)/m_I=\chi(V_f^I)/(m_I\vert G/G\cap G_f^I\vert)$.

Let $I$ be a proper subset of $I_0=\{1, \cdots, n\}$ (i.e.\ $I\ne \emptyset$, $I\ne I_0$),
and let $\overline{I}=I_0\setminus I$. If $({\rm supp\,}f)\cap Z^I$ consists of less than $\vert I\vert$
points, i.e.\ if $f$ has less than $\vert I\vert$ monomials in the variables $x_i$ with $i\in I$,
then $\chi(V_f^I)=0$ (e.g. due to the Varchenko formula \cite{Varch}) and therefore
$\zeta_{\widehat{h}_f^I}(t)=1$, $\zeta^{{\rm orb}, I}_{f,G}(t)=1$. In this case
$({\rm supp\,}\widetilde{f})\cap Z^{\overline{I}}$ consists of less than $\vert \overline{I}\vert$
points and therefore $\zeta^{{\rm orb}, I}_{\widetilde{f},\widetilde{G}}(t)=1$.

Let $\vert({\rm supp\,}f)\cap Z^I\vert=\vert I\vert$. From Equation~(\ref{zeta-quotient})
and Lemma~\ref{1-tm} it follows that
\begin{equation}
 \zeta^{{\rm orb}, I}_{f,G}(t)=\left(1-t^{{\rm lcm}(m_I,k_I)}\right)^{s'_I}\,,
\end{equation}
where
$$
k_I=\frac{\vert G\cap G_f^I\vert}{\vert G\cap G_f^I\cap {\rm SL}(n,\CC)\vert}\,.
$$ Therefore
\begin{equation}
 \zeta^{{\rm orb}, I}_{f,G}(t)=\left(1-t^{\ell_I}\right)^{s'_I}\,,
\end{equation}
where 
\begin{equation}\label{lcm}
 \ell_I={\rm lcm}\left(\frac{\vert G+ G_f^I+ G_{f,0}\vert}{\vert G+ G_f^I\vert},
 \frac{\vert G\cap G_f^I\vert}{\vert G\cap G_f^I\cap {\rm SL}(n,\CC)\vert}\right)\,.
\end{equation}

In this case $\vert({\rm supp\,}\widetilde{f})\cap Z^{\overline{I}}\vert=
\vert \overline{I}\vert$ and therefore
$$
 \zeta^{{\rm orb}, \overline{I}}_{\widetilde{f},\widetilde{G}}(t)=
 \left(1-t^{\widetilde{\ell}_I}\right)^{\widetilde{s}'_{\overline{I}}}\,,
$$
where 
\begin{equation}\label{lcm2}
 \widetilde{\ell_I}=
 {\rm lcm}
 \left(
 \frac{\vert \widetilde{G}+ G_{\widetilde{f}}^{\overline{I}}+ G_{\widetilde{f},0}\vert}
 {\vert \widetilde{G}+ G_{\widetilde{f}}^{\overline{I}}\vert},
 \frac{\vert \widetilde{G}\cap G_{\widetilde{f}}^{\overline{I}}\vert}
 {\vert \widetilde{G}\cap G_{\widetilde{f}}^{\overline{I}}\cap {\rm SL}(n,\CC)\vert}
 \right)\,.
\end{equation}
According to \cite[Lemma 1]{BLMS}, one has $G_{\widetilde{f}}^{\overline{I}}=\widetilde{G_f^I}$;
by Lemma~\ref{SL}, one has $\widetilde{G_{\widetilde{f},0}}= G_f \cap{\rm SL}(n,\CC)$
and $\widetilde{G_{f,0}}= G_{\widetilde{f}} \cap {\rm SL}(n,\CC)$.
This means that the subgroup $G+ G_f^I+ G_{f,0}\subset G_f$ is dual to 
$\widetilde{G}\cap G_{\widetilde{f}}^{\overline{I}}\cap {\rm SL}(n,\CC)\subset G_{\widetilde{f}}$
and the subgroup $G+ G_f^I\subset G_f$ is dual to
$\widetilde{G}\cap G_{\widetilde{f}}^{\overline{I}}\subset G_{\widetilde{f}}$.
Therefore
$$
\frac{\vert G+ G_f^I+ G_{f,0}\vert}{\vert G+ G_f^I\vert}=
\frac{\vert \widetilde{G}\cap G_{\widetilde{f}}^{\overline{I}}\vert}
{\vert \widetilde{G}\cap G_{\widetilde{f}}^{\overline{I}}\cap {\rm SL}(n,\CC)\vert}\,.
$$
In the same way
$$
\frac{\vert G\cap G_f^I\vert}{\vert G\cap G_f^I\cap {\rm SL}(n,\CC)\vert}=
\frac{\vert \widetilde{G}+ G_{\widetilde{f}}^{\overline{I}}+ G_{\widetilde{f},0}\vert}
 {\vert \widetilde{G}+ G_{\widetilde{f}}^{\overline{I}}\vert}
$$
and therefore $\ell_I=\widetilde{\ell}_{\overline{I}}$.
In \cite{MMJ} it was shown that
$\ell_I s'_I=(-1)^n\widetilde{\ell}_{\overline{I}}{\widetilde{s}_{\overline{I}}}'$.
Thus $s'_I=(-1)^n \widetilde{s}'_{\overline{I}}$. Therefore the factor
$\zeta^{{\rm orb}, I}_{f,G}(t)$ in Equation~(\ref{product}) for $\overline{\zeta}^{{\rm orb}}_{f,G}(t)$
is equal to the factor
$\left(\zeta^{{\rm orb}, \overline{I}}_{\widetilde{f},\widetilde{G}}(t)\right)^{(-1)^n}$
in the corresponding equation for
$\left(\overline{\zeta}^{{\rm orb}}_{\widetilde{f},\widetilde{G}}(t)\right)^{(-1)^n}$.

Now let $I=I_0$. One has $G_f^{I_0}=\{0\}$ and therefore
$\zeta^{{\rm orb}, I_0}_{f,G}(t)=\zeta_{\widehat{h}_f^{I_0}}(t)=\left(1-t^{m_{I_0}}\right)^{s_{I_0}}$,
where $m_{I_0}=\vert G_{f,0}/G\cap G_{f,0}\vert=\frac{\vert G+G_{f,0}\vert}{\vert G\vert}$. On the other hand,
by Lemma~\ref{1-tm}, one has 
$$
\prod\limits_{g\in\widetilde{G}}(1-t)_g=(1-t^{\widetilde{k}})^{\widetilde{r}},
$$
where
$$
\widetilde{k}=\frac{\vert\widetilde{G}\vert}{\vert \widetilde{G}\cap {\rm SL}(n,\CC)\vert}.
$$
Due to Lemma~\ref{SL}, the subgroup $\widetilde{G}\cap {\rm SL}(n,\CC)\subset G_{\widetilde{f}}$
is dual to the subgroup $G+G_{f,0}\subset G_f$.
Therefore $m_{I_0}=\widetilde{k}$. In~\cite{MMJ} it was shown that 
$m_{I_0}s_{I_0}=(-1)^{n-1}\widetilde{k}\widetilde{r}$. Therefore $s_{I_0}=(-1)^{n-1}\widetilde{r}$
and the factor
$\zeta^{{\rm orb}, I_0}_{f,G}(t)$ in Equation~(\ref{product}) for $\overline{\zeta}^{{\rm orb}}_{f,G}(t)$
is equal to the factor
$(\prod\limits_{g\in\widetilde{G}}(1-t)_g)^{(-1)^{n-1}}$
in the corresponding equation for
$(\overline{\zeta}^{{\rm orb}}_{\widetilde{f},\widetilde{G}}(t))^{(-1)^n}$.
\end{proof}

\begin{remark}
 In Equation~(\ref{lcm}) for the exponent $\ell_I$, the first argument of the least common multiple
 is connected with the monodromy action and the second one with the age shift.
 The duality interchanges this numbers. The one for the pair $(f,G)$ connected with the
 monodromy action is equal to the one for the dual pair $(\widetilde{f},\widetilde{G})$ connected with the
 age shift and vice versa (see~(\ref{lcm2})).
\end{remark}


\bigskip
\noindent Leibniz Universit\"{a}t Hannover, Institut f\"{u}r Algebraische Geometrie,\\
Postfach 6009, D-30060 Hannover, Germany \\
E-mail: ebeling@math.uni-hannover.de\\

\medskip
\noindent Moscow State University, Faculty of Mechanics and Mathematics,\\
Moscow, GSP-1, 119991, Russia\\
E-mail: sabir@mccme.ru

\end{document}